\documentclass[12pt,twoside,reqno]{amsart}
\theoremstyle{plain}

\newtheorem*{theorem*}{Theorem}
\newtheorem*{corollary}{Corollary}
\newtheorem*{lemma}{Lemma}
\newcommand{\C}{\mathbb{C}}

\newcommand{\ie}{{\it{i.$\,$e.\ }}}

\newcommand{\re}{\operatorname{Re}}

\newcommand{\Log}{\mathrm{Log}}

\begin{document}
\title[The Multiplicative anomaly of several operators]{The Multiplicative anomaly  of three or more commuting elliptic operators}

\author{Victor Castillo-Garate}
\address{Mathematics Department, 202 Mathematical Sciences Bldg., University
of Missouri Columbia, MO 65211, USA}
\email{vactn7@mail.missouri.edu}

\author{Eduardo Friedman}
\address{Departamento de Matem\'aticas,
 Universidad de Chile, Casilla 653, Santiago, Chile}
\email{friedman@uchile.cl}

\author{Marius M\u antoiu}
\address{Departamento de Matem\'aticas,
 Universidad de Chile, Casilla 653, Santiago, Chile}
\email{mantoiu@uchile.cl}

\thanks{This work was partially supported by the Chilean Programa Iniciativa
 Cient\'{\i}fica Milenio grant ICM P07-027-F and Fondecyt grants 1110277 and 1120300.}

\begin{abstract}$\zeta$-regularized determinants are well-known to fail to be multiplicative, so that in general
 $\det_\zeta(AB)\not= \det_\zeta(A)\det_\zeta(B)$. Hence one is lead  to  study   the $n$-fold multiplicative
 anomaly
 $$
 M_n(A_1,...,A_n) :=\frac{\det_\zeta\!\!\Big(\!\prod_{i=1}^n A_i\Big)}{\prod_{i=1}^n \det_\zeta(A_i)}
 $$
attached to $n$ (suitable) operators $A_1,...,A_n$. We show that if the  $A_i$ are   commuting pseudo-differential
elliptic operators, then
their joint multiplicative anomaly can be expressed in terms of the pairwise multiplicative anomalies. Namely
$$
M_n(A_1,...,A_n)^{m_1+\cdots+m_n} =\prod_{1\le i<j\le n}M_2(A_i,A_j)^{m_i+m_j},
$$
where $m_j$ is the order of  $A_j$. The proof relies on Wodzicki's 1987 formula for the pairwise multiplicative anomaly $M_2(A,B)$ of two commuting elliptic operators.
\end{abstract}

\maketitle

\section{Introduction}
\noindent For an important class of operators $A$, one can define its $\zeta$-regularized determinant
 as
$$
\textstyle{\det_\zeta(A)}:=\exp\!\big(\!-\frac{d}{ds}\zeta_A(s)\big|_{s=0}\big),
$$
where $\zeta_A(s):=\sum_i \lambda_i^{-s}$ is the spectral  zeta function of $A$, extended to $s=0$ by
analytic continuation \cite{Se}.
Although such determinants have played an important role in    mathematical physics, geometry and number theory \cite{El1} \cite{El2} \cite{KV} \cite{JL}, it has long been known that
they fail to be multiplicative, \ie even for commuting operators $\det_\zeta(AB)\not= \det_\zeta(A)\det_\zeta(B)$,  in general.

This phenomenon has lead to  the  study of the multiplicative (or determinant) anomaly
$$
M_2(A,B) :=\frac{\det_\zeta(AB)}{\det_\zeta(A)\det_\zeta(B)}.
$$
A  formula for $M_2(A,B)$ was given by Wodzicki \cite{Wo} \cite[\S6]{Ka}. He assumed    $A$ and $B$ are
 commuting, positive, invertible, elliptic self-adjoint pseudo-differential operators of positive order acting on   the space of  smooth sections   of a  finite-dimensional complex vector bundle $E$ over a compact   $C^\infty$-manifold  $M$ without boundary. Here we have fixed a  Hermitian metric on $E$ and  a density on $M$.
Under these assumptions Wodzicki's formula reads \cite{Ka}
\begin{equation}\label{WODD}
\log\! \big(M_2(A,B)\big) = \frac{\mathrm{res}\big(\Log^2(\sigma_{A,B})\big)}{2\,\mathrm{ord}\,A\,
\mathrm{ord}\,B \,(\mathrm{ord}\,A +\mathrm{ord}\,B) },
\end{equation}
where
\begin{equation*}
\sigma_{A,B}:=A^{\mathrm{ord}\,B}B^{-\mathrm{ord\,A}},
\end{equation*}
  $\mathrm{ord}\,A$ is the order of $A$, and res denotes the Wodzicki residue.

Even with Wodzicki's formula, the  multiplicative anomaly $M_2(A,B)$ attached
 to pairs of commuting operators is in general   difficult to compute. It has been explicitly
computed in terms of special functions only for a handful of
cases (see \cite[\S2.3]{El2} and the references there). Perhaps for this reason the joint
multiplicative  anomaly
 $$
 M_n(A_1,...,A_n) :=\frac{\det_\zeta\!\!\Big(\!\prod_{i=1}^n A_i\Big)}{\prod_{i=1}^n \det_\zeta(A_i)}
 $$
attached to $n$  commuting operators $A_1,...,A_n$ seems  not to have been studied. There is a trivial reduction
$$
M_n(A_1,...,A_n)=M_{n-1}(A_1A_2,A_3,...,A_n)M_2(A_1,A_2)
$$
which can be unwound inductively into a formula for $M_n$ in terms of $M_2$'s, but it would be hardly practical as all of the $A_i$'s
are simultaneously involved in some of the resulting $M_2$'s.

We show that there is a simple formula expressing   $M_n(A_1,...,A_n)$ in terms of  the individual  $M_2(A_i,A_j)$.

\begin{theorem*} Suppose  $A_1,...,A_n$ are   $n$ commuting,  positive, invertible, elliptic self-adjoint pseudo-differential operators of positive order acting  on the  smooth sections of a  finite-dimensional vector bundle $E$ over a compact manifold $M$ without boundary. Then,   their joint multiplicative anomaly $M_n$  is defined for  $n\ge2$ and can be expressed in terms of the pairwise multiplicative anomalies $M_2$ as
\begin{equation}\label{main}
M_n(A_1,...,A_n)^{m_1+\cdots+m_n} =\prod_{1\le i<j\le n}M_2(A_i,A_j)^{m_i+m_j},
\end{equation}
where $m_i$ is the order of $A_i$.
\end{theorem*}
\noindent Our proof uses some elementary identities involving  the operator $\Log^2(\sigma_{A,B})$ appearing inside the Wodzicki residue in \eqref{WODD}.
A special
case of the above theorem was proved in \cite{CGF}.

The  theorem reduces the calculation of $M_n$ to that of $M_2$. In fact,  we can also reduce to  $M_k$ for  any  integer $k $  between 2 and $n$.

\begin{corollary} For  $2\le k\le n$, and letting  $C_q^p:=\frac{p!}{q!(p-q)!}$, we have \vskip.05cm
\begin{equation*}
M_n(A_1,\dots,A_n)^{(m_1+\dots+m_n)C^{n-2}_{k-2}}=\!\prod_{1\le i_1<\dots< i_k\le n}\!\!M_k(A_{i_1},\dots,A_{i_k})^{m_{i_1}+\dots+m_{i_k}}.
\end{equation*}
\end{corollary}
We shall prove the corollary at the end of the next section.

\section{Proofs}

Let $M$ be a compact smooth $C^\infty$-manifold provided with a 1-density, and let $E/M$ be an finite-dimensional
complex vector bundle over $M$ endowed with a Hermitian metric. Let $A$ be a pseudo-differential operator acting
on the $C^\infty$-sections of $E/M$. We can extend $A$ to a possibly unbounded operator on the Hilbert space of square-integrable
sections of $E/M$.
We shall say that $A$ satisfies Wodzicki's
hypothesis if $A$ is a positive, invertible, elliptic self-adjoint pseudo-differential operator of positive order.
  Then we can define the spectral zeta function
\begin{equation*}
 \zeta_A(s):=\sum_i \lambda_i^{-s} ,\qquad\qquad(\re(s)>m/ \mathrm{ord}\,A)
\end{equation*}
where $\lambda_i$ runs  over the (positive) eigenvalues of $A$ and the real branch of $\log$ is used to define the complex powers \cite{Se}. The spectral zeta function admits a meromorphic continuation to $\C$, regular at $s=0$, so we can
define the $\zeta$-regularized determinant of $A$ by
\begin{equation*}
 \textstyle{\det_\zeta(A)}:=\exp\!\big(\!-\zeta_A^\prime(0)\big).
\end{equation*}

If $A_1,...,A_n$ are $n$ commuting   operators  satisfying Wodzicki's hypothesis, their product $A:=A_1A_2\cdots \,A_n$ also satisfies it, so we can define
\begin{equation*}
\delta_n=\delta_n(A_1,...,A_n):= -\zeta_{A}^\prime(0)+\sum_{i=1}^n \zeta_{A_i}^\prime(0).
\end{equation*}
The  joint  multiplicative anomaly of the $A_i$ is then
\begin{equation*}
M_n=M_n(A_1,...,A_n):= \exp(\delta_n)=\frac{\det_\zeta(A_1A_2\cdots \, A_n)}{\det_\zeta(A_1)\det_\zeta(A_2)\cdots\,\det_\zeta(A_n)}.
\end{equation*}

We will prove the relation \eqref{main} between $M_n$ and the various $M_2$'s by induction on $n$. For this our
main tools will be the trivial reduction formula
 \begin{equation}\label{Red}
 \delta_n(A_1,...,A_n)= \delta_{n-1}( A_1A_2,...,A_n)+\delta_2(A_1,A_2),
\end{equation}
and
 Wodzicki's formula\footnote{\ Wodzicki has not published his proof, although he kindly sketched it to us in a letter. A proof can be found in \cite[p.\ 726]{Ok}.}
 \begin{equation}\label{Wod}
\delta_2(A,B) = \frac{\mathrm{res}\big(\Log^2(\sigma_{A,B})\big)}{2\,\mathrm{ord}\,A\,
\mathrm{ord}\,B \,(\mathrm{ord}\,A +\mathrm{ord}\,B) },
\end{equation}
where
\begin{equation*}
\sigma_{A,B}:=A^{\mathrm{ord}\,B}B^{-\mathrm{ord\,A}},
\end{equation*}
  $\mathrm{ord}\,A$ is the order of $A$, and res denotes the Wodzicki residue.
In fact, we shall need to know nothing about the Wodzicki residue beyond the fact that it is linear.
Instead, we shall rely on  some simple   properties of the operator $\Log $ acting on commuting self-adjoint operators.

We begin by noting that $\delta_2(A,B)$ can be expressed  in terms of $\delta_2$ of two operators having the same order. Namely,
\begin{equation}\label{C:Reducc grados ig}
(\mathrm{ord}\,A\,+\,\mathrm{ord}\,B)\, \delta_2(A,B)= 2 \delta_2(A^{\mathrm{ord}\,B},B^{\mathrm{ord}\,A}).
\end{equation}
The proof is immediate from Wodzicki's formula \eqref{Wod} and the linearity of
Wodzicki's residue.

The next calculation will be the main step in our inductive proof of the Theorem stated in \S1.
\begin{lemma}  Let $A_1,A_2,\dots,A_n$  be $n$ commuting operators, $n\ge3$, all satisfying Wod\-zicki's hypotheses, and set $m_i:=\mathrm{ord}\,A_i$. Then
\begin{align} \nonumber
     \sum_{1\leq{i}<{j}\leq{n}}  \delta_2(A_i^{m_j},A_j^{m_i}) = \sum_{j=2}^{n-1} & \, \delta_2\big((A_1A_2)^{m_{j+1}},A_{j+1}^{m_1+m_2}\big)  \,  +\,\sum_{3\leq{i}<{j}\leq{n}} {\delta_2(A_i^{m_j},A_j^{m_i})}\\ &+ \frac{m_1+\cdots+m_n}{2}\delta_2(A_1,A_2)  \label{lemma}.
\end{align}
\end{lemma}
\begin{proof}
Since
\begin{align*}
\sum_{1\leq{i}<{j}\leq{n}} \delta_2(A_i^{m_j},A_j^{m_i})  -    \sum_{3\leq{i}<{j}\leq{n}}  \delta_2(A_i^{m_j},&A_j^{m_i})\\ & =\sum_{j=2}^n \delta_2(A_1^{m_j},A_j^{m_1})  + \sum_{j=3}^{n} \delta_2(A_2^{m_j},A_j^{m_2}),
\end{align*}
it suffices to prove
\begin{multline}\label{E:Probar inducc}
      \sum_{j=2}^{n-1} \delta_2\big((A_1A_2)^{m_{j+1}},A_{j+1}^{m_1+m_2}\big)
         + \frac{m_1+\cdots+m_n}{2}\delta_2(A_1,A_2)   \\
 = \sum_{j=2}^{n} \delta_2(A_1^{m_j},A_j^{m_1})+ \sum_{j=3}^{n} \delta_2(A_2^{m_j},A_j^{m_2}).
\end{multline}
We first consider $n=3$. Then \eqref{E:Probar inducc} reads
\begin{multline}\label{n3}
       \delta_2\big((A_1A_2)^{m_3},A_3^{m_1+m_2})+\frac{m_1+m_2+m_3}{2}\delta_2(A_1,A_2)\\ = \delta_2(A_1^{m_2},A_2^{m_1}) +\delta_2(A_1^{m_3},A_3^{m_1})
    +\delta_2(A_2^{m_3},A_3^{m_2}).
\end{multline}
Using  \eqref{C:Reducc grados ig}, after some simple cancellations  we find that to prove \eqref{n3} we must prove
\begin{multline}\label{E:cuadrado}
     (m_1+m_2+m_3)\delta_2(A_1A_2,A_3)+m_3 \delta_2(A_1,A_2)\\ =(m_1+m_3)\delta_2(A_1,A_3)+(m_2+m_3)\delta_2(A_2,A_3).
\end{multline}
In view of Wodzicki's formula \eqref{Wod},  we  compute
\begin{align*}
 &\frac{m_1+m_2+m_3}{2(m_1+m_2)m_3(m_1+m_2+m_3)}\,\Log^2 \big(       (A_1A_2)^{m_3} A_3^{-m_1-m_2}     \big)
 \\ & \qquad      \qquad   \qquad  \qquad  \qquad  +
\frac{m_3}{2m_1m_2(m_1+m_2)}\,\Log^2 \big(     A_1^{m_2} A_2^{-m_1}      \big)\\ &
\\
  & \   =\frac{1}{2(m_1+m_2)} \Big( \frac{\big(    m_3(\Log\,A_1+\Log\,A_2) -(m_1+m_2)\Log\,A_3  \big)^2 }{m_3}
  \\ & \qquad  \qquad  \qquad  \qquad  \qquad  \qquad  \qquad  \qquad  \qquad
          \quad +\frac{m_3\big(  m_2\Log\,A_1-m_1\Log\,A_2    \big)^2   }{m_1m_2} \Big)   \\ &    \\
   & \ =  \frac{(m_3\Log\,A_1-m_1\Log\,A_3)^2}{2m_1m_3}
              + \frac{(m_3\Log\,A_2 - m_2\Log\,A_3)^2}{2m_2m_3}\quad [\text{to check this step,}
              \\ & \qquad\qquad \text{ compare coefficients  of } \Log\,A_i\,\Log\,A_j \text{ on both sides for } 1\le i\le j\le3]   \\
    &\  = \frac{m_1+m_3}{2m_1m_3(m_1+m_3)}   \Log^2 \big(   A_1^{m_3} A_3^{-m_1}      \big)  +\frac{m_2+m_3}{2m_2m_3(m_2+m_3)}  \Log^2 \big(    A_2^{m_3} A_3^{-m_2}     \big).
\end{align*}
If we now  apply the Wodzicki residue to the above equation,  formula \eqref{Wod} and linearity  of the residue  yield \eqref{n3}.
This proves the case $n=3$.

We can now complete the proof of the lemma by induction on $n$.
Comparing \eqref{E:Probar inducc} for $n$ and $n+1$, we find that the inductive step amounts to
 $$ \delta_2((A_1A_2)^{m_{n+1}},A_{n+1}^{m_1+m_2})+\frac{m_{n+1}}{2}\delta_2(A_1,A_2)=
 \delta_2(A_1^{m_{n+1}},A_{n+1}^{m_1})+\delta_2(A_2^{m_{n+1}},A_{n+1}^{m_2}). $$
Using \eqref{C:Reducc grados ig}, we see that the above is exactly \eqref{E:cuadrado}, with $A_3$ replaced by $A_{n+1}$. \end{proof}

We  now prove the theorem stated in \S1,  in the equivalent form
\begin{align}\label{TheCore}
\delta_n(A_1,...,A_n)=\sum_{1\le i<j\le n} \frac{m_i+m_j}{m_1+\dots+m_n}\delta_2(A_i,A_j).
\end{align}
We again proceed by induction on $n$. For $n=2$ both sides of \eqref{TheCore}   are trivially equal, so we suppose $n\ge3$.  By the inductive hypothesis and the equal-orders formula \eqref{C:Reducc grados ig}, we have
\begin{align*}
\delta_{n-1}(A_1A_2,A_3,&\dots,A_n)\\ & = 2\frac{\sum_{j=2}^{n-1} \delta_2\big((A_1A_2)^{m_{j+1}},A_{j+1}^{m_1+m_2}\big)  +  \sum_{3\leq{i}<{j}\leq{n}} \delta_2(A_i^{m_j},A_j^{m_i})}  {(m_1+m_2)+\dots+m_n}.
\end{align*}
 Substituting this into the trivial reduction formula \eqref{Red}  we find
\begin{eqnarray*}
  \delta_n(A_1,A_2,\dots,A_n) &=&  \frac{2}{m_1+m_2+\cdots+m_n} \bigg (  \sum_{j=2}^{n-1} {\delta_2\big((A_1A_2)^{m_{j+1}},A_{j+1}^{m_1+m_2}\big)} \\
    && \quad +  \sum_{3\leq{i}<{j}\leq{n}} {\delta_2(A_i^{m_j},A_j^{m_i})} +\frac{m_1+\dots+m_n}{2}\delta_2(A_1,A_2) \bigg )
    \\
    &=&  \frac{2}{m_1+m_2+\cdots+m_n} \sum_{1\leq{i}<{j}\leq{n}} {\delta_2(A_i^{m_j},A_j^{m_i})}  \quad \quad  [\text{use }\eqref{lemma}]
     \\ &=&  \frac{1}{m_1+\cdots+m_n} \sum_{1\leq{i}<{j}\leq{n}}(m_i+m_j)  \delta_2(A_i ,A_j ) \qquad  [\text{use \eqref{C:Reducc grados ig}}],
\end{eqnarray*}
which concludes the proof of  \eqref{TheCore}.\qed

\medskip

We now prove the corollary stated at the end of \S1.
It suffices to prove,  for $2\le k\le n$,
\begin{equation}\label{mihaela}
(m_1+\dots+m_n)C^{n-2}_{k-2}\delta_n(A_1,\dots,A_n)=\!\sum_{1\le i_1<\dots< i_k\le n}\!\!(m_{i_1}+\dots+m_{i_k})\delta_k(A_{i_1},\dots,A_{i_k}).
\end{equation}
Set $\omega:=\{1,\dots\,n\}$. For a subset $\gamma=\{i_1,\dots,i_\ell\}\subset\omega$ of cardinality $\#\gamma=\ell$, set
$$
\mu(\gamma):=(m_{i_1}+\dots+m_{i_\ell})\delta_\ell(A_{i_1},\dots,A_{i_\ell}),
$$
which is unambiguously defined due to the symmetry of $\delta_\ell$ for commuting operators.
In this notation,  \eqref{mihaela}  amounts to
\begin{equation*}
C^{n-2}_{k-2}\,\mu(\omega)=\sum_{\beta\subset\omega,\#\beta=k}\!\mu(\beta).
\end{equation*}
Our main theorem can be rewritten as
\begin{equation}\label{toma}
\mu(\beta)=\sum_{\alpha\subset\beta,\#\alpha=2}\!\mu(\alpha).
\end{equation}
We have, using \eqref{toma},
\begin{equation*}
\sum_{\beta\subset\omega,\#\beta=k}\mu(\beta)=\sum_{\beta\subset\omega,\#\beta=k}\,\sum_{\alpha\subset\beta,\#\alpha=2}\mu(\alpha)
=C^{n-2}_{k-2}\!\sum_{\alpha\subset\omega,\#\alpha=2}\mu(\alpha),
\end{equation*}
since every two-element subset  $\alpha\subset\omega=\{1,\dots,n\}$ is contained in exactly $C_{k-2}^{n-2}$ subsets $\beta\subset\omega$ of cardinality $k$. We conclude the proof by again using \eqref{toma}, with $\beta=\omega$.
\qed

\end{document}